\documentclass[12pt]{article}
\usepackage{amsthm,amsmath,graphicx,fullpage}
\newtheorem{lemma}{Lemma}

\newtheorem{thm}{Theorem}

\begin{document}
			
		\title{Chain lengths in the type $B$ Tamari lattice}				
		\author{Edward Early and Stephanie Thrash}
		\maketitle
		
\begin{abstract}
	We find the largest union of two chains in the type $B$ Tamari lattice by generalizing 
	the techniques used for the classical (type $A$) Tamari lattice with a description of 
	the type $B$ case due to Hugh Thomas.
\end{abstract}
				
The Tamari lattice is a partially ordered set originally defined in terms of binary-operation bracketings~\cite{ht}. 
Its elements are enumerated by Catalan numbers, so it is not surprising that there are several equivalent 
descriptions. One description comes from the Symmetric group in a way that can be generalized to other 
Coxeter groups, giving rise to Cambrian lattices~\cite{re}. Our focus is on the type $B$ Tamari lattice that 
thus arises. A description of the elements of this poset due to Thomas~\cite{th} allows us to use similar 
techniques to those employed in classical Tamari lattice to find the largest union of two chains~\cite{ee}.

The elements of the Tamari lattice $T_n$ can be represented by $n$-tuples $(v_1,\dots,v_n)$ of integers from 1 to $n$ 
following two rules~\cite{ht}:\\
(\textit{i}) each $v_i\ge i$\\
(\textit{ii}) if $i\le j\le v_i$, then $v_j\le v_i$.

The partial ordering is then componentwise comparison, so the technical complexity of the description of the 
elements yields a very simple ordering. For the purposes of this paper, the only thing we need to know about 
the type $B$ Tamari lattice $T_n^B$ is that its elements can be represented by $n$-tuples $(r_1,\dots,r_n)$ of symbols from 
$\{0,1,2,\dots,n-1,\infty\}$ following two rules~\cite{th}:\\
(\textit{i}) for $i<j$, $r_i\le r_j-(j-i)$ if $r_j-(j-i)$ is nonnegative\\
(\textit{ii}) if $\infty>r_i\ge i$, then $r_{n+i-r_i}=\infty$.

While the description of the elements in this case is evidently more complicated, the ordering remains 
componentwise comparison. The bottom element of the poset will be $(0,\dots,0)$ and the top 
element will be $(\infty,\dots,\infty)$. The first property implies that any time an entry $r_j=k$ 
for a finite number $k$, the prior $k$ entries must also be finite and are at most $0,1,\dots,k-1$ 
in that order. In particular, consecutive entries can only be equal if they are both 0 
or both $\infty$, and an $\infty$ can only be followed by 0 or $\infty$. The second property says 
that large entries in the $n$-tuple force infinities elsewhere. For example, anything that starts 
with 1 must end with $\infty$. The largest element with all finite entries is $(0,1,\dots,n-1)$.
Hasse diagrams for $T_2^B$ and $T_3^B$ are shown in Figure~\ref{figtb23}.

\begin{figure}[htb]
\centering
\includegraphics[height=1.5in]{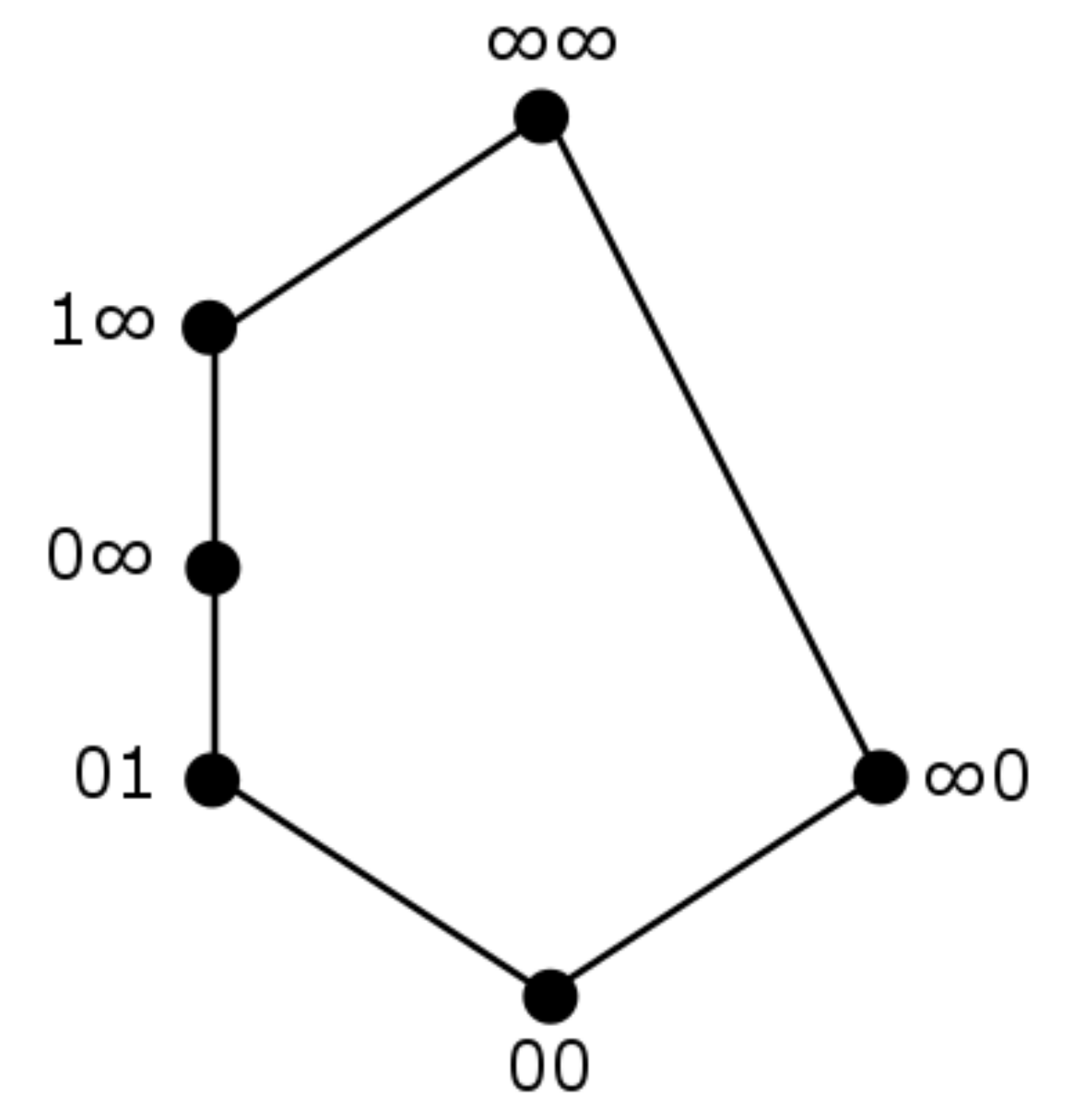} \hskip1in \includegraphics[height=2in]{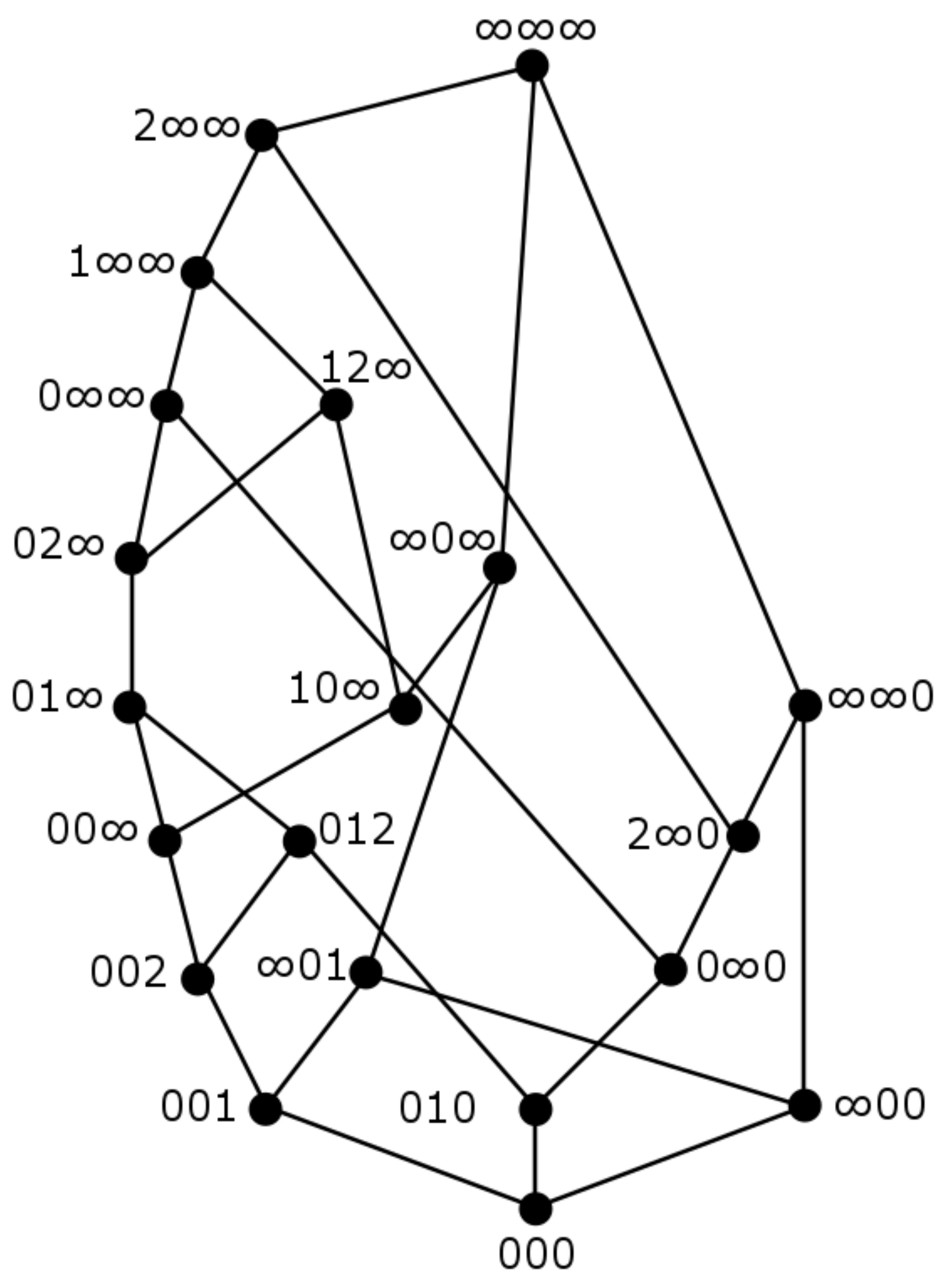}
\caption{$T_2^B$ and $T_3^B$}
\label{figtb23}
\end{figure}

Given any poset $P$, there exists a partition $\lambda(P)$ such that the sum of the first $k$ 
parts of $\lambda(P)$ is the maximum number of elements in a union of $k$ chains in $P$. A 
theorem of Greene and Kleitman states that the conjugate of $\lambda(P)$ has the same property 
for antichains in $P$~\cite{g,gk}. Let $\lambda_k(P)$ denote the $k$th part of this partition.

The poset $T_n^B$ is not graded, so the key to constructing two disjoint chains of maximal length is 
to extract the ``leveled" elements. A poset is \textit{leveled} if every element (except those on 
the top and bottom levels) covers something on the level below and is covered by something on the 
level above. Every element in a leveled poset $P$ is on one of the longest chains of $P$, but this 
is a slightly weaker condition than being a graded poset where every maximal chain must have the same 
length.

Consider the subposet of elements of $T_n^B$ that are on a chain of maximal length. For such an element 
$\mu$, define $r(\mu)$ to be the length of the longest chain from the bottom element to $\mu$. 

\begin{lemma}
Leveled elements appear on level $r(\mu)$, which is the sum of the entries of $\mu$ counting 
each $\infty$ as $n$. Unleveled elements drawn at the lowest possible level will be at or 
below the level of their sum.
\label{lem1}
\end{lemma}

\begin{proof}
Among the leveled elements we can construct at least one chain where each step increases only one entry by exactly 1. 
An explicit description of such a chain will be given in the proof of Theorem~\ref{thm1}. For the unleveled elements, 
every covering relation must involve increasing at least one part by at least 1, so the sum provides an upper bound 
on the maximum number of coverings needed to reach that element from the bottom.
\end{proof}

The longest chain therefore has length $n^2$, so $\lambda_1(T_n^B)=n^2+1$.
Our main result is the following.

\begin{thm}
For $n\ge4$, $\lambda_2(T_n^B)=\lambda_1(T_n^B)-5$.
\label{thm1}
\end{thm}

\begin{proof}
First we construct two disjoint chains of the desired lengths, then show that no larger union is possible.
The first chain starts at $(0,\dots,0)$ and increases the rightmost possible part by 1, except that adding 
1 to $n-1$ yields $\infty$. For example, when $n=4$ the chain is 
(0,0,0,0), (0,0,0,1), (0,0,0,2), (0,0,0,3), (0,0,0,$\infty$), (0,0,1,$\infty$), (0,0,2,$\infty$), 
(0,0,3,$\infty$), (0,0,$\infty$,$\infty$), (0,1,$\infty$,$\infty$), (0,2,$\infty$,$\infty$), (0,3,$\infty$,$\infty$), 
(0,$\infty$,$\infty$,$\infty$), (1,$\infty$,$\infty$,$\infty$), (2,$\infty$,$\infty$,$\infty$), (3,$\infty$,$\infty$,$\infty$), 
($\infty$,$\infty$,$\infty$,$\infty$).

The second chain begins at $(0,\dots,0,1,2)$ and increases the leftmost possible part by 1, ending at 
$(n-2,n-1,\infty,\dots,\infty)$. When $n=4$, the chain is 
(0,0,1,2), (0,0,1,3), (0,0,2,3), (0,1,2,3), (0,1,2,$\infty$), (0,1,3,$\infty$), (0,2,3,$\infty$), 
(1,2,3,$\infty$), (1,2,$\infty$,$\infty$), (1,3,$\infty$,$\infty$), (2,3,$\infty$,$\infty$).

These two chains are disjoint because for larger values of $n$ the elements that are on the same level 
from each chain will have a different number of 0 and/or $\infty$ entries. In 
particular, the first chain gains an $\infty$ on every $n$th level. Up to level $n$ the first chain has 
only one non-zero entry while the second chain always has at least two nonzero entries.
The second chain acquires its first 
$\infty$ at level $\frac{n(n-1)}{2}+1$, which is greater than $2n$ for $n\ge5$ and has only one 0 at that point, 
which becomes a 1 $n-1$ levels later. 
The first chain has at least one 0 until level $n^2-n+1$. At that point all but one entry is $\infty$, while the 
second chain always has at least two finite entries, thus establishing that the two chains are disjoint.

This second chain is six elements smaller than the first 
rather than the five needed for the theorem. We append the unleveled element $(0,\dots,0,1,0)$ at 
the beginning of the second chain to make it one longer, constructively proving that 
$\lambda_2(T_n^B)\ge\lambda_1(T_n^B)-5$.

To show that two chains cannot contain any more elements it suffices (by the Greene-Kleitman theorem) to 
show that $T_n^B$ can be partitioned into $n^2+1$ antichains, five of which comprise a single element. 
To that end, we start by observing that a Hasse diagram provides a decomposition into antichains via 
its levels. To get the desired antichains, start with every element placed on the lowest possible level. 
All of the unleveled elements can be shifted up one level as long as they are all moved together. This 
makes the bottom two levels comprise one element each, namely $(0,\dots,0)$ and $(0,\dots,0,1)$.
The top element $(\infty,\dots,\infty)$ covers elements with only one finite entry. One of these is 
the leveled element $(n-1,\infty,\dots,\infty)$ and the rest contain a 0 because no other finite number 
can be preceded by $\infty$. The unleveled elements covered by the top element are therefore started at least 
$n$ levels down by Lemma~\ref{lem1}, so even moving them up 1 still keeps them out of the top three levels 
for $n\ge4$. The elements covered by $(n-1,\infty,\dots,\infty)$ are similarly the leveled element 
$(n-2,\infty,\dots,\infty)$ and those that replace an $\infty$ with a 0 and thus start too many levels down 
to interfere. We therefore have the antichains to show that $\lambda_2(T_n^B)\le\lambda_1(T_n^B)-5$, which 
together with the chain construction completes the proof.
\end{proof}

An unlabeled Hasse diagram of $T_4^B$ is shown in Figure~\ref{figtb4}. The only leveled elements in this 
poset are the ones used in the construction.

\begin{figure}[htb]
	\centering
	\includegraphics[height=3in]{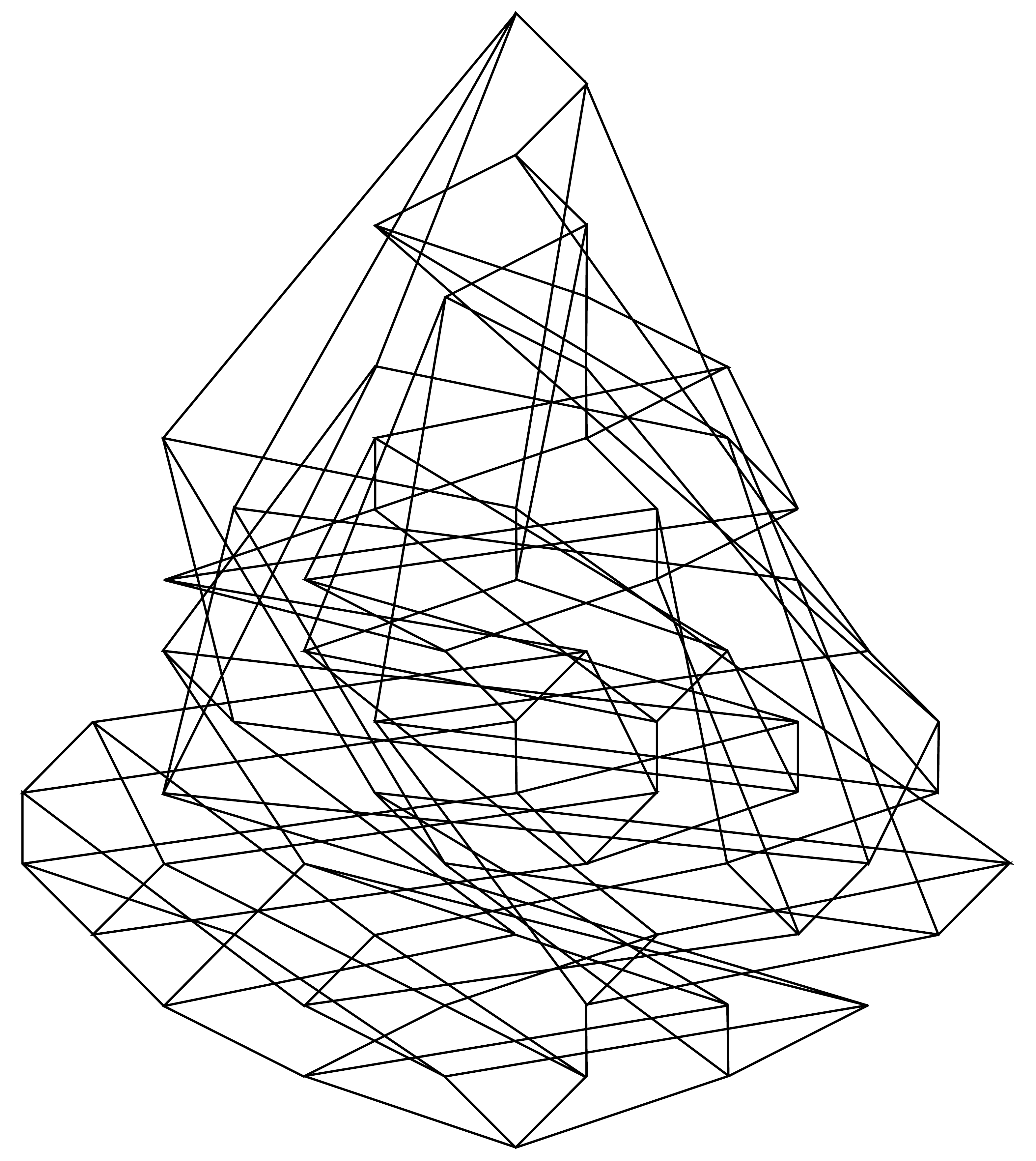}
	\caption{$T_4^B$}
	\label{figtb4}
\end{figure}

Unlike $T_n$, $T_n^B$ is not self-dual, hence the odd difference in chain lengths, though the leveled elements   
are self-dual. A third long disjoint chain begins to emerge at $n=5$, where there are six levels of one 
element each and four levels of two elements each within the leveled elements. It appears likely that 
$\lambda_k(T_n^B)-\lambda_{k+1}(T_n^B)$ will be constant for fixed $k$ and large $n$, so computing that 
value for $k>1$ is a natural followup problem. Another area exploration is the type $D$ analogue.

\section*{Acknowledgements}
The authors thank Nathan Reading for suggesting this problem and John Stembridge, whose posets package for 
Maple made Figure~\ref{figtb4} feasible.

\end{document}